\documentclass[12pt]{amsart}

\usepackage[mathscr]{eucal}
\usepackage{amsmath,amssymb,amscd,amsthm}%,latexsym
\usepackage{color}
\usepackage{amsmath,amsthm,amssymb,graphicx}
\usepackage{epsfig}
\newtheorem{theorem}{Theorem}

\newtheorem{corollary}{Corollary}

\theoremstyle{definition}

\theoremstyle{plane}

\def \beq{ \begin{equation} }
\def \eeq{\end{equation}}

\usepackage[top=3.8cm, bottom=3.8cm, left=3.5cm, right=3.5cm]{geometry}
\title{Rectangular orbits of the curved 4-body problem}
\begin{document}
\maketitle
\markboth{Florin Diacu and Brendan Thorn}{Rectangular orbits of the curved $4$-body problem}
\author{\begin{center}
{\bf Florin Diacu}$^{1,2}$ and {\bf Brendan Thorn}$^2$\\
\smallskip
{\footnotesize $^1$Pacific Institute for the Mathematical Sciences\\
and\\
$^2$Department of Mathematics and Statistics\\
University of Victoria\\
P.O.~Box 3060 STN CSC\\
Victoria, BC, Canada, V8W 3R4\\
diacu@uvic.ca and bthorn@uvic.ca\\
}\end{center}

}

\vskip0.5cm

\begin{center}
\today
\end{center}

\begin{abstract}
We consider the $4$-body problem in spaces of constant curvature and study the existence of spherical and hyperbolic rectangular solutions, i.e.\ equiangular quadrilateral motions on spheres and hyperbolic spheres. 
We focus on relative equilibria (orbits that 
maintain constant mutual distances) and rotopulsators (configurations that rotate and change size, but preserve equiangularity). We prove that when such orbits exist, they are necessarily spherical or hyperbolic squares, i.e.\ equiangular equilateral quadrilaterals.
\end{abstract}

%%%%%%%%
\section{Introduction}
%%%%%%%%

We consider the curved 4-body problem, given by the differential equations that describe the motion of 4 point masses in spaces of constant Gaussian curvature: 2- and 3-spheres for positive curvature and hyperbolic 2- and 3-spheres for negative curvature. Using suitable transformations, the study of the problem can be reduced to  $\mathbb S^2, \mathbb S^3, \mathbb H^2$, and $\mathbb H^3$, respectively, \cite{Diacu03}. These equations provide a natural extension of the classical Newtonian equations from Euclidean space. The curved 2-body problem has a long history, starting with Bolyai, Lobachevsky, Dirichlet, Lipschitz, Killing, Liebmann, and others, \cite{Diacu03}, \cite{Diacu04}. The general case of N bodies, $N\ge 3$, has been recently developed in a suitable framework, \cite{Diacu01}, \cite{Diacu02}, \cite{Diacu03}, \cite{Diacu04}, \cite{Diacu05},
\cite{Diacu06}, \cite{Diacu07}, \cite{Diacu08}, \cite{Diacu09},
\cite{Diacu10}, \cite{Perez}, \cite{Tibboel}.

The goal of this paper is to study the existence of classes of solutions of the curved 4-body problem that (with one exception) maintain spherical or hyperbolic rectangular configurations (i.e.\ form equiangular quadrilaterals) all along the motion.  Our results prove that when such orbits exist, they are necessarily spherical or hyperbolic squares, i.e.\ equiangular equilateral quadrilaterals.

We will be interested in both {\it relative equilibria}, i.e.\ orbits for which the mutual distances remain constant during the motion, and {\it rotopulsators}, i.e.\ systems that rotate and change size, but maintain a spherical or hyperbolic equiangular shape. We take exception from restricting ourselves to rectangles when studying relative equilibria along great circles of $\mathbb S^2$. Since antipodal configurations are singular for the equations of motion, we consider trapezoidal configurations in that case.

In Section 2, we introduce the equations of motion. In Section 3, we show that spherical trapezoidal (non-rectangular) relative equilibria do not exist when the bodies move along a great circle of $\mathbb S^2$. The nonexistence of spherical and hyperbolic rectangular non-square relative equilibria in $\mathbb S^2$ and $\mathbb H^2$ is proved in Section 4. In Sections 5 and 6 we show that when spherical and hyperbolic rectangular rotopulsators and relative equilibria exist in $\mathbb S^3$ and $\mathbb H^3$, they are necessarily spherical or hyperbolic squares,
a case in which the masses must be equal.

%%%%%%
%%%%%%  Equations of motion
%%%%%%

\section{Equations of motion}

Consider $4$ point particles (bodies) of masses $m_1, m_2, m_3, m_4>0$ moving in $\mathbb S^2$ or
$\mathbb{S}^3$ (thought as embedded in the ambient Euclidean space $\mathbb R^3$ or $\mathbb R^4$, respectively) or in $\mathbb H^2$ or $\mathbb H^3$ (embedded in the ambient Minkowski space $\mathbb R^{2,1}$ or $\mathbb R^{3,1}$, respectively),
where
$$
{\mathbb S}^3=\{(w,x,y,z)\ | \ w^2+x^2+y^2+z^2=1\}, \ \mathbb S^2=\{\mathbb S^3 \ {\rm with}\ w=0\},
$$
$$
{\mathbb H}^3=\{(w,x,y,z)\ | \ w^2+x^2+y^2-z^2=-1, \ z>0\}, \ \mathbb H^2=\{\mathbb H^3 \ {\rm with}\ w=0\}.
$$
Then the configuration of the system is described by the vector
$$
{\bf q}=({\bf q}_1,{\bf q}_2, {\bf q}_3,{\bf q}_4),
$$
where ${\bf q}_i=(w_i,x_i,y_i,z_i)$ in 3D and ${\bf q}_i=(x_i,y_i,z_i)$ in 2D,\ $ i=1,2,3,4$, denote the position vectors of the bodies.  The equations of motion (see \cite{Diacu04} or \cite{Diacu09} for their derivation using constrained Lagrangian dynamics) are given by the system
\begin{equation}
\label{both}
\ddot{\bf q}_i=\sum_{j=1,j\ne i}^N\frac{m_j[{\bf q}_j-\sigma({\bf q}_i\cdot {\bf q}_j){\bf q}_i]}{[\sigma-\sigma({\bf q}_i\cdot {\bf q}_j)^2]^{3/2}}-\sigma(\dot{\bf q}_i\cdot \dot{\bf q}_i){\bf q}_i, \ \ i=1,2,3,4,
\end{equation}
with initial-condition constraints
\begin{equation}
({\bf q}_i\cdot{\bf q}_i)(0)=\sigma, \ \ ({\bf q}_i\cdot\dot{\bf q}_i)(0)=0, \ \ i=1,2,3,4,
\end{equation}
where $\cdot$ is the standard inner product of signature $(+,+,+,+)$ in $\mathbb S^3\subset\mathbb R^4$ and $(+,+,+)$ in $\mathbb S^2\subset\mathbb R^3$, but the Lorentz inner product of signature $(+,+,+,-)$ in $\mathbb H^3\subset\mathbb R^{3,1}$ and $(+,+,-)$ in $\mathbb H^2\subset\mathbb R^{2,1}$, and
$\sigma=\pm 1$, depending on whether the curvature is positive or negative. Relative to the inner product, the gravitational force acting on each body has a tangential component (the above sum) and a radial component (the term involving the velocities).

From Noether's theorem, system \eqref{both} has the energy integral,
$$
T({\bf q},\dot{\bf q})-U({\bf q})=h,
$$
where 
$$U({\bf q})=\sum_{1\le i<j\le 4}\frac{\sigma m_im_j{\bf q}_i\cdot{\bf q}_j}{[\sigma-\sigma({\bf q}_i\cdot{\bf q}_j)^2]^{1/2}}$$ 
is the force function ($-U$ representing the potential),
$$
T({\bf q},\dot{\bf q})=\frac{1}{2}\sum_{i=1}^4m_i(\dot{\bf q}_i\cdot\dot{\bf q}_i)
(\sigma{\bf q}_i\cdot{\bf q}_i)
$$
is the kinetic energy, with $h$ representing an integration constant; and the integrals of the total angular momentum,
$$
\sum_{i=1}^4m_i{\bf q}_i\wedge\dot{\bf q}_i={\bf c},
$$
where $\wedge$ is the wedge product and ${\bf c}=(c_{wx},c_{wy},c_{wz},c_{xy},c_{xz},c_{yz})$ denotes an integration vector in 3D, whereas ${\bf c}=(c_{xy},c_{xz},c_{yz})$ is the integration vector in 2D, each component measuring the rotation of the system about the origin of the frame relative to the plane corresponding to the bottom indices (in 2D it is customary to express the rotation relative to an axis orthogonal to this plane).

%%%%%%%%%
%%%%%%%%%  Trapezoid
%%%%%%%%%
\section{Trapezoidal fixed points in $\mathbb S^2$}

In this section we study the motion of the 4 bodies along great circles of $\mathbb S^2$. Notice, however, that the equations of motion lose meaning if at least a pair of bodies are antipodal, therefore rectangular orbits along great circles of $\mathbb S^2$ cannot exist. We will therefore check whether trapezoidal orbits occur. Of course, thought on $\mathbb S^2$, this figure would be degenerate, therefore we prefer to regard it as an Euclidean trapezoid in the plane of the equator, $z=0$. Due to symmetries, it is natural to restrict to the case when the masses on each of the parallel sides of the trapezoid are equal. We can now prove the following result, which refers to fixed points, i.e.\ static configurations.

%%%%  Trapezoid Thm.
\begin{theorem}
In the curved $4$-body problem in $\mathbb S^2$, there are no fixed points inscribed in any great circle, such that the bodies form a trapezoid with equal masses on each of the parallel sides, i.e.\  $m_1=m_2:=m>0$ and $m_3=m_4:=M>0$.
\end{theorem}
\begin{proof}
Without loss of generality, consider an arbitrary non-rectangular trapezoid inscribed in the equator, $z=0$. Let the $y$-axis bisect the parallel chords connecting the bodies, and assume the shorter parallel chord to be on the positive $y$ side. All angles are measured from the positive $x$-axis. By thus fixing the coordinate system, each body $m_i$, $i=1,2,3,4$, is located in quadrant $i$ at initial position ${\bf q}_i(0)=(x_i(0),y_i(0))$. Then
\begin{equation}\label{results}
x_1(0)=-x_2(0),\ \ x_3(0)=-x_4(0),\ \ y_1(0)=y_2(0),\ \ y_3(0)=y_4(0).
\end{equation}
To form a fixed point of system \eqref{both}, the bodies $m_i$, $i=1,2,3,4$, must satisfy the initial conditions
$$
\ddot {\bf q}_i(0)= \dot {\bf q}_i(0)=0,\ \ i=1,2,3,4.
$$
Using equations \eqref{both} and eliminating the duplications given by linearly dependent equations, we are  led to the linear homogeneous system
\begin{equation}\label{final-sys}
{\bf A x}={\bf 0},
\end{equation}
where ${\bf A}=
\begin{bmatrix} \label{final-mat}
  q_{12}     & q_{13}+q_{14}\\ 
  q_{31}+q_{32} & q_{34} 
\end{bmatrix},$  ${\bf x}=(m,M)^T$, and, for $i,j=1,2,3,4,i\neq j$,
\begin{equation}
\label{expressions}
q_{ij}=\frac{x_j-a_{ij}x_i}{(1-a_{ij}^2)^{3/2}},\ \ a_{ij}=x_ix_j+y_iy_j.
\end{equation}
To prove the result, we will show that $\det{\bf A}<0$, where
\begin{equation}\label{determinant}
\det {\bf A}= q_{12}q_{34}-(q_{13}+q_{14})(q_{31}+q_{32})= q_{12}q_{34}-q_{13}q_{31}-q_{13}q_{32}-q_{14}q_{32}-q_{14}q_{31}. 
\end{equation}
For this, we first express $\det{\bf A}$ in terms of
\eqref{expressions}, and then convert the expression to polar coordinates that avoid any rectangular configuration, i.e.\ take
$$
x_1=\cos\alpha,\ y_1=\sin\alpha,\ x_3=\cos\beta,\ y_3=\sin\beta,
$$
with $\alpha\in(0,\pi/2)$ and $\beta\in(\pi,3\pi/2)$, but such that
$\beta-\alpha\ne\pi$. Some long but straightforward computations show that
\begin{equation}\label{pol-det}
\begin{split}
\det{\bf A}=&
\frac{\cos\alpha\cos\beta}{16|\cos\alpha|^3|\cos\beta|^3|\sin\alpha||\sin\beta|}+\\
&\frac{4\sin^2\alpha \sin^2\beta\cos\alpha\cos\beta[\sin^4(\alpha+\beta)-\sin^4(\alpha-\beta)]}{\sin^4(\alpha-\beta)\sin^4(\alpha+\beta)}.
\end{split}
\end{equation}
It is easy to see that, on the above domains of $\alpha$ and $\beta$, $\det{\bf A}$ is always negative, so the only solution of system \eqref{final-sys} is the trivial one, $m=M=0$. 
\end{proof}

Since relative equilibria can be generated only from fixed-point configurations by giving the particle system a rotation (see \cite{Diacu04} for a proof), an obvious consequence of this result is the following statement.

%%% corollary
\begin{corollary}
In the curved $4$-body problem in $\mathbb S^2$, there are no trapezoidal relative equilibria generated from fixed points
if the masses on each of the parallel sides are equal.
\end{corollary}

%%%%%%%
%%%%%%%   RE 2D
%%%%%%%
\section{Rectangular relative equilibria in $\mathbb S^2$ and $\mathbb H^2$}

Since rectangular relative equilibria might exist in $\mathbb S^2$ if they stay away from great circles, we will further check this possibility. In $\mathbb H^2$ we have no restrictions such as the one for great circles in $\mathbb S^2$, so we treat the positive and the negative curvature simultaneously.

%%%%  Rectangle 2D
\begin{theorem}
In the curved $4$-body problem in $\mathbb S^2$ and $\mathbb H^2$, the only spherical and hyperbolic rectangular relative equilibria of equal masses, $m_1, m_2, m_3, m_4=:m>0$, occur when the rectangular is a square.
\end{theorem}
\begin{proof}
We will show that the only possible spherical and hyperbolic rectangular relative equilibria of equal masses in $\mathbb S^2$ and $\mathbb H^2$ are the squares, solutions whose existence was already proved in \cite{Diacu02}.

Notice first that we can write the angles between the particles and the $x$-axis viewed from the centre of the (Euclidean) rectangle in the plane $z=c,\ c\in(-1,0)\cup(0,1)$ as
$$
 \alpha_1=\alpha,\ \ \alpha_2=\pi-\alpha,\ \ \alpha_3=\pi+\alpha,\ \ \alpha_4=-\alpha,\ \ \alpha\in(0,\pi/2), 
$$
for the respective $m_i, \ i=1,2,3,4$. We are seeking solutions of system \eqref{both} of the form
\begin{equation}\label{releq_sols}
\begin{split}
{\bf q}=({\bf q}_1, {\bf q}_2, {\bf q}_3, {\bf q}_4), {\bf q}_i=(x_i,y_i,z_i),\hspace{3.8cm} \\
x_i=r\cos(\omega t+\alpha_i),\ y_i=r\sin(\omega t+\alpha_i),\ z_i=z=(1-r^2)^{1/2},\ \  i=1,2,3,4,
\end{split}
\end{equation}
where $\omega$ is the angular velocity of the point masses and $r$ is the radius of the circles formed by the intersection of $\mathbb S^2$ or $\mathbb H^2$ with the plane that contains the rectangle, which must be different from $z=0$.
Substituting a candidate solution of the form \eqref{releq_sols} into system \eqref{both}, we are led to the equations
\begin{equation}\label{Amega_both}
\begin{split}
&\omega^2=m\frac{[-1-(\sigma {\bf q}_1\cdot {\bf q}_2)]}{(\sigma r^2-1)[\sigma-\sigma(\sigma {\bf q}_1\cdot {\bf q}_2)^2]^{3/2}}+
\\&m\frac{[-1-(\sigma {\bf q}_1\cdot {\bf q}_3)]}{(\sigma r^2-1)[\sigma-\sigma(\sigma {\bf q}_1\cdot {\bf q}_3)^2]^{3/2}}+m\frac{[1-(\sigma {\bf q}_1\cdot {\bf q}_4)]}{(\sigma r^2-1)[\sigma-\sigma(\sigma {\bf q}_1\cdot {\bf q}_4)^2]^{3/2}},
\end{split}
\end{equation}
\begin{equation}\label{Bmega_both}
\begin{split}
&\omega^2=m\frac{[1-(\sigma {\bf q}_1\cdot {\bf q}_2)]}{(\sigma r^2-1)[\sigma-\sigma(\sigma {\bf q}_1\cdot {\bf q}_2)^2]^{3/2}}+
\\&m\frac{[-1-(\sigma {\bf q}_1\cdot {\bf q}_3)]}{(\sigma r^2-1)[\sigma-\sigma(\sigma {\bf q}_1\cdot {\bf q}_3)^2]^{3/2}}+m\frac{[-1-(\sigma {\bf q}_1\cdot {\bf q}_4)]}{(\sigma r^2-1)[\sigma-\sigma(\sigma {\bf q}_1\cdot {\bf q}_4)^2]^{3/2}},
\end{split}
\end{equation}
which must be simultaneously satisfied. By subtracting \eqref{Amega_both} from \eqref{Bmega_both}, we obtain
$$
[\sigma-\sigma(\sigma(1-2r^2\cos^2\alpha))^2]^{3/2}=[\sigma-\sigma(\sigma(1-2r^2\sin^2\alpha))^2]^{3/2},
$$
which is equivalent to $\cos^2\alpha = \sin^2\alpha$,
therefore satisfied only for $\alpha = \pi/4$.
\end{proof}

%%%%%%%%
%%%%%%%%
%%%%%%%%
\section{Rectangular orbits in $\mathbb S^3$}

In this section we consider the two possible types of
spherical rectangular equal-mass orbits in $\mathbb S^3$, the positive elliptic, which have one elliptic rotation, and the positive elliptic-elliptic, which have two elliptic rotations, \cite{Diacu05}. In the former case we prove that rotopulsators of this kind exist only when the rectangle is a square. In the latter case we show that these orbits must be square relative equilibria. 

%%%%%%%%%%
\begin{theorem}[{\bf Positive elliptic rectangular rotopulsators}] In the curved $4$-body problem in $\mathbb S^3$, the only spherical rectangular positive elliptic equal-mass rotopulsators are the spherical square rotopulsators. The orbit rotates relative to the plane $wx$, but has no rotation with respect to the planes $wy, wz, xy, xz$, and $yz$.
\end{theorem}
\begin{proof}
Consider equal masses, $m_1=m_2=m_3=m_4:=m>0$, and a candidate solution (see Definition 1 in \cite{Diacu05}) of the form
\begin{equation}\label{rpesol}
\begin{gathered}
{\bf q}=({\bf q}_1, {\bf q}_2, {\bf q}_3, {\bf q}_4),\ {\bf q}_i=(w_i,x_i,y_i,z_i),\ i=1,2,3,4,\\
w_1=r\cos\alpha, \ x_1=r\sin\alpha, \ y_1=y,\ z_1=z,\\
w_2=r\cos(\alpha+\theta),\ x_2=r\sin(\alpha+\theta),\ y_2=y, \ z_2=z,\\
w_3=r\cos(\alpha+\pi),\ x_3=r\sin(\alpha+\pi),\ y_3=y, \ z_3=z,\\
w_4=r\cos(\alpha+\theta+\pi),\ x_4=r\sin(\alpha+\theta+\pi), \ y_4=y,\ z_4=z,\\
\end{gathered}
\end{equation}
with $r,\alpha, y, z$ functions of $t$, with $r^2+y^2+z^2=1$, and $\theta\ne 0$ is a constant that gives the angle between $m_1$ and $m_2$ viewed from the centre of the corresponding Euclidean triangle. A solution of this form would obviously maintain a rectangular configuration for all time since, if $\epsilon_{ij}={\bf q}_i\cdot{\bf q}_j$, $i,j\in\{1,2,3,4\}, i\ne j$, then
$$
\epsilon_{12}=\epsilon_{34}=(1-y^2-z^2)\cos\theta+y^2+z^2 = (1-\delta z^2)\cos\theta + \delta z^2,
$$
$$
\epsilon_{13}=\epsilon_{24}=-(1-y^2-z^2)+y^2+z^2=-1+2\delta z^2,
$$
$$\epsilon_{14}=\epsilon_{23}=-(1-y^2-z^2)\cos\theta+y^2+z^2=(-1+\delta z^2)\cos\theta+\delta z^2.
$$
Moreover,
$$
\dot\alpha=\frac{c}{4m(1-y^2-z^2)}=\frac{c}{4m(1-\delta z^2)},
$$
where $\delta=\gamma^2+1$, and $\gamma=y/z$. Using Criterion 1 in \cite{Diacu05}, it follows that $r, y$, and $z$ must satisfy the system
\begin{equation}\label{diff1}
\begin{cases}
r\ddot\alpha+2\dot r\dot\alpha=0\cr
\ddot y=F(y,z,\dot y,\dot z)y\cr
\ddot z=F(y,z,\dot y,\dot z)z,
\end{cases}
\end{equation}
where
$$
F(y,z,\dot y,\dot z)=m\bigg[\frac{1-[(1-\delta z^2)\cos\theta+\delta z^2]}{(1-[(1-\delta z^2)\cos\theta+\delta z^2]^2)^{3/2}} +
\frac{1-(-1+2\delta z^2)}{(1-[-1+2\delta z^2]^2)^{3/2}}+
$$
$$
\frac{1-[(-1+\delta z^2)\cos\theta+\delta z^2]}{(1-[(-1+\delta z^2)\cos\theta+\delta z^2]^2)^{3/2}}\bigg]-
\frac{\dot y^2 + \dot z^2 - (y \dot z - \dot y z)^2}{1-\delta z^2}-\frac{c^2}{16m^2(1-\delta z^2)},
$$
and that the identity  
\begin{equation}\label{identity-pe}
\frac{mr\sin\theta}{(\epsilon_{12}^2-1)^{3/2}}=\frac{mr\sin\theta}{(\epsilon_{14}^2-1)^{3/2}}
\end{equation}
must take place. But \eqref{identity-pe} is valid only for $\theta=\pi/2$, so if solutions of the form \eqref{rpesol} exist, they must be spherical squares.

The first equation in \eqref{diff1} is identically satisfied, and from the second and third we can conclude that $y \ddot z =\ddot y z$, which implies
$
y \dot z- \dot y z = k\ {\rm (constant)}.
$
However, from the angular momentum integrals we have that $4m(y \dot z -\dot y z)=c_{yz}$, so it follows that $k=c_{yz}/4m$. But a straightforward computation shows that $c_{yz}=0$, which implies that $y \dot z -\dot y z = 0$ and, therefore, $\frac{d}{dt}\frac{y}{z}=0$ for $z\neq0$, so $\gamma=y(t)/z(t)$ is a constant. Next, we note that
\begin{equation}
\begin{split}
&\sin\alpha+\sin(\alpha+\theta)+\sin(\alpha+\pi)+\sin(\alpha+\theta+\pi)=
\\&\cos\alpha+\cos(\alpha+\theta)+\cos(\alpha+\pi)+\cos(\alpha+\theta+\pi)=0,
\end{split}
\end{equation}
from which we can conclude that $c_{wy}=c_{wz}=c_{xy}=c_{xz}=0$. This means that the rectangle has no rotation relative to any plane apart from the $wx$-plane. The energy relation 
takes the form
\begin{equation}\label{energy}
\begin{split}
&h=\frac{2m[\dot y^2 + \dot z^2 - (y \dot z - \dot y z)^2]}{1-\delta z^2}+\frac{c^2}{8m(1-\delta z^2)}-
\\&2m^2\bigg[\frac{q_{12}}{(1-q_{12}^2)^{1/2}}+\frac{q_{13}}{(1-q_{13}^2)^{1/2}}+\frac{q_{14}}{(1-q_{14}^2)^{1/2}}\bigg],
\end{split}
\end{equation}
and we can now write $F$ as
$$
F(z)=m\bigg[\frac{1-2[(1-\delta z^2)\cos\theta+\delta z^2]+[(1-\delta z^2)\cos\theta+\delta z^2]^3}{(1-[(1-\delta z^2)\cos\theta+\delta z^2]^2)^{3/2}}
$$
$$
+\frac{1-2(-1+2\delta z^2)+(-1+2\delta z^2)^3}{(1-(-1+2\delta z^2)^2)^{3/2}}
$$
$$
+\frac{1-2[(-1+\delta z^2)\cos\theta+\delta z^2]+[(-1+\delta z^2)\cos\theta+\delta z^2]^3}{(1-[(-1+\delta z^2)\cos\theta+\delta z^2]^2)^{3/2} }\bigg]-\frac{h}{2m}.
$$
So, since $y(t)$ and $z(t)$ are proportional to each other, we can reduce system \eqref{diff1} to
\begin{equation}\label{diff3}
\begin{cases}
\dot z = \nu \cr
\dot \nu = F(z)z.
\end{cases}
\end{equation}
Thus, by applying the standard existence and uniqueness theorems, we see there exists a large class of analytic positive elliptic square rotopulsators in $\mathbb S^3$ for all admissible values of the involved parameters.
\end{proof}

%%%%%%%%%%%%
\begin{theorem}[{\bf Positive elliptic-elliptic rectangular orbits}] In the curved $4$-body problem in $\mathbb S^3$, there are no equal-mass positive elliptic-elliptic rectangular rotopulsators or relative equilibria.

\end{theorem}
\begin{proof}
Consider equal masses, $m_1=m_2=m_3=m_4:=m>0$
and a candidate solution (see Definition 2 in \cite{Diacu05}) of the form
\begin{equation}\label{rpeesol}
\begin{gathered}
{\bf q}=({\bf q}_1, {\bf q}_2, {\bf q}_3, {\bf q}_4), \ {\bf q}_i=(w_i,x_i,y_i,z_i),\ i=1,2,3,4,\\
w_1=r\cos\alpha,\ x_1=r\sin\alpha,\ y_1=\rho\cos\beta,\ z_1=\rho\sin\beta,\\
w_2=r\cos\alpha,\ x_2=r\sin\alpha,\ y_2=-\rho\cos\beta,\ z_2=-\rho\sin\beta,\\
w_3=r\cos(\alpha+a),\ x_3=r\sin(\alpha+a),\ y_3=\rho\cos(\beta+b),\ z_3=\rho\sin(\beta+b),\\
w_4=r\cos(\alpha+a),\ x_4=r\sin(\alpha+a), y_4=-\rho\cos(\beta+b),\ z_4=-\rho\sin(\beta+b),\\
\end{gathered}
\end{equation}
where $r, \rho, \alpha, \beta$ are functions of $t$, with $r^2+\rho^2=1$, and $a,b\ne 0$ are constants. This candidate solution would maintain a rectangular configuration all along the motion since the quantities $\epsilon_{ij}={\bf q}_i \cdot {\bf q}_j$, $i,j\in\{1,2,3,4\}, i\neq j$, are
\begin{equation*}
\begin{split}
\epsilon_{12}=\epsilon_{34}=2r^2-1,\ \
\epsilon_{14}=\epsilon_{23}=r^2\cos a -\rho^2\cos b,\\
\epsilon_{13}=\epsilon_{24}=r^2\cos a+\rho^2\cos b,\ \ \ \ \ \ \ \ \ \ \ \ \  {}
\end{split}
\end{equation*}
as long as the diagonals, corresponding to
$\epsilon_{13}$ and $\epsilon_{24}$, are longer than
the sides, which depend on $\epsilon_{12}, \epsilon_{34}, \epsilon_{14}$, and $\epsilon_{23}$. Further computations show that
$$
\dot\alpha=\frac{c_1}{4mr^2},\ \ \dot\beta=\frac{c_2}{4m\rho^2},
$$
where $c_1=c_{wx}\ne 0$ and $c_2=c_{yz}\ne 0$. Using Criterion 2 in \cite{Diacu05}, we can conclude that the equations of motion contain, among others, the equations
\begin{equation}\label{diff2rpee}
\begin{cases}
r\ddot\alpha+2\dot r\dot\alpha=\frac{mr\sin a}{[1-(r^2\cos a + (1-r^2)\cos b)^2]^{3/2}}+
\frac{mr\sin a}{[1-(r^2\cos a - (1-r^2)\cos b)^2]^{3/2}}\cr
r\ddot\alpha+2\dot r\dot\alpha=-\frac{mr\sin a}{[1-(r^2\cos a +(1-r^2)\cos b)^2]^{3/2}}-
\frac{mr\sin a}{[1-(r^2\cos a - (1-r^2)\cos b)^2]^{3/2}}\cr
\rho\ddot\beta+2\dot\rho\dot\beta=\frac{m\rho\sin b}{[1-((1-\rho^2)\cos a +\rho^2\cos b)^2]^{3/2}}-
\frac{m\rho\sin b}{[1-((1-\rho^2)\cos a - \rho^2\cos b)^2]^{3/2}}\cr
\rho\ddot\beta+2\dot\rho\dot\beta=-\frac{m\rho\sin b}{[1-((1-\rho^2)\cos a + \rho^2\cos b)^2]^{3/2}}+
\frac{m\rho\sin b}{[1-((1-\rho^2)\cos a - \rho^2\cos b)^2]^{3/2}}.\cr
\end{cases}
\end{equation}
The above expressions of $\dot\alpha$ and $\dot\beta$ show that the left hand sides in all equations must be zero,
so the system is satisfied only for $a=0,\pi$ and $b=\pm\pi/2$. In all cases it follows that $\epsilon_{14}=\epsilon_{23}=\epsilon_{24}=\epsilon_{13}$, which implies that the diagonals are equal with two of the sides. But this is not a proper spherical rectangle, so solutions of this type don't exist. 
\end{proof}

%%%%%%%%%%
%%%%%%%%%%
%%%%%%%%%%
\section{Rectangular rotopulsators in $\mathbb H^3$}

In this final section we show that the only hyperbolic rectangular rotopulsators of equal masses in $\mathbb H^3$ are the negative elliptic square rotopulsators (see Definition 3 in \cite{Diacu05}). For this we also prove that negative hyperbolic and negative elliptic-hyperbolic rectangular rotopulsators (see Definitions 4 and 5 in \cite{Diacu05}) do not exist.

%%%%%%%%%
\begin{theorem}[{\bf Negative elliptic rectangular rotopulsators}] In the curved $4$-body problem in $\mathbb H^3$, the only negative elliptic rectangular rotopulsators are the negative elliptic square rotopulsators. Moreover, these orbits rotate relative to the $wx$ plane, but have no rotation with respect to the planes $wy, wz, xy, xz,$ and $yz$.
\end{theorem}
\begin{proof}
Consider a candidate solution (see Definition 3 in \cite{Diacu05}) of the form
\begin{equation}\label{rnesol}
\begin{gathered}
{\bf q}=({\bf q}_1, {\bf q}_2, {\bf q}_3, {\bf q}_4), \ {\bf q}_i=(w_i,x_i,y_i,z_i),\ i=1,2,3,4,\\
w_1=r\cos\alpha,\ x_1=r\sin\alpha,\ y_1=y,\ z_1=z,\\
w_2=r\cos(\alpha+\theta),\ x_2=r\sin(\alpha+\theta),\ y_2=y,\ z_2=z,\\
w_3=r\cos(\alpha+\pi),\ x_3=r\sin(\alpha+\pi),\ y_3=y,\ z_3=z,\\
w_4=r\cos(\alpha+\theta+\pi),\ x_4=r\sin(\alpha+\theta+\pi),\ y_4=y,\ z_4=z,\\
\end{gathered}
\end{equation}
where $r, \alpha, y, z$ are functions of $t$, with $r^2+y^2-z^2=-1$, and $\theta\ne 0$ is a constant that measures the angle between $m_1$ and $m_2$ viewed from the centre of the Euclidean rectangle. This candidate solution maintains a rectangular configuration for all time since, if $\mu_{ij}={\bf q}_i\cdot{\bf q}_j,\ i,j\in\{1,2,3,4\}, \ i\ne j$, then
$$
\begin{gathered}
\mu_{12}=\mu_{34}=r^2\cos\theta-r^2-1=(z^2-y^2-1)\cos\theta+y^2-z^2,\\
\mu_{13}= \mu_{24}=-2r^2-1=2z^2-2y^2-3,\\
\mu_{14}=\mu_{23}=-r^2\cos\theta-r^2-1=-(z^2-y^2-1)\cos\theta+y^2-z^2.
\end{gathered}
$$
Moreover,
$$
\dot\alpha=\frac{b}{4mr^2}.
$$
Criterion 3 in \cite{Diacu05} leads us to the system
\begin{equation}\label{rnediff1}
\begin{cases}
\ddot y=G(y,z,\dot y,\dot z)y\cr
\ddot z=G(y,z,\dot y,\dot z)z,\cr
r\ddot\alpha+2\dot r\dot\alpha=\frac{mr\sin\theta}{[(r^2\cos\theta-r^2-1)^2-1]^{3/2}}-\frac{mr\sin\theta}{[(-r^2\cos\theta-r^2-1)^2-1]^{3/2}}\cr
r\ddot\alpha+2\dot r\dot\alpha=-\frac{mr\sin\theta}{[(r^2\cos\theta-r^2-1)^2-1]^{3/2}}+\frac{mr\sin\theta}{[(-r^2\cos\theta-r^2-1)^2-1]^{3/2}},
\end{cases}
\end{equation}
where
\begin{equation}\label{rneG1}
\begin{split}
G(y,z,\dot y,\dot z)=&m\bigg[\frac{1+\mu_{12}}{(\mu_{12}^2-1)^{3/2}}+\frac{1+\mu_{13}}{(\mu_{13}^2-1)^{3/2}}+\frac{1+\mu_{14}}{(\mu_{14}^2-1)^{3/2}}\bigg]+
\\&\frac{\dot z^2 - \dot y^2 + (y \dot z - \dot y z)^2}{z^2-y^2-1}+\frac{b^2}{16m^2(z^2-y^2-1)}.
\end{split}
\end{equation}
The above expression of $\dot\alpha$ shows that the left hand sides of the last two equations in system \eqref{rnediff1} are zero, so those equations are identically satisfied only if $\theta=\pm\pi/2$, which implies that the candidate solution must be a spherical square.

From \eqref{rnediff1}, we know that $y \ddot z =\ddot y z$, which implies
$
y \dot z- \dot y z = k\ {\rm (constant)}.
$
However, from the integrals of angular momentum we have that $4m(y \dot z -\dot y z)=c_{yz}$, so it follows that $k=c_{yz}/4m$. A simple computation shows that $c_{yz}=0$, therefore $y \dot z -\dot y z = 0$, so $\frac{d}{dt}\frac{y}{z}=0$ for $z\neq0$. (Moreover, $c_{wx}\ne 0$ and $c_{wy}=c_{wz}=c_{xy}=c_{xz}=0.$) Consequently $z(t)=\gamma y(t)$, where $\gamma$ is constant, so we can reduce \eqref{rnediff1} to the system
\begin{equation}\label{rnediff3}
\begin{cases}
\dot z = \nu \cr
\dot \nu = G(y,z,\dot y,\dot z)z.
\end{cases}
\end{equation}
Standard results of the theory of differential equations show the existence of solutions for admissible initial value problems for the above system, so the negative elliptic spherical square candidates are the only kinds spherical rectangular solutions.
\end{proof}

%%%%%%%%
\begin{theorem}[{\bf Negative hyperbolic rectangular rotopulsators}] In the curved $4$-body problem in $\mathbb H^3$, there are no negative hyperbolic rectangular rotopulsators.
\end{theorem}
\begin{proof}
Consider a solution candidate (see Definition 4 in \cite{Diacu05}) of the form
\begin{equation}\label{rnhsol}
\begin{gathered}
{\bf q}=({\bf q}_1, {\bf q}_2, {\bf q}_3, {\bf q}_4),\ {\bf q}_i=(w_i,x_i,y_i,z_i), \ i=1,2,3,4,\\
w_1=w,\ x_1=x, \ y_1=\eta\sinh\beta,\ z_1=\eta\cosh\beta,\\
w_2=-w,\ x_2=-x,\ y_2=\eta\sinh\beta,\ z_2=\eta\cosh\beta,\\
w_3=w,\ x_3=x, \ y_3=\eta\sinh(\beta+\phi),\ z_3=\eta\cosh(\beta+\phi),\\
w_4=-w,\ x_4=-x,\ y_4=\eta\sinh(\beta+\phi),\ z_4=\eta\cosh(\beta+\phi),\\
\end{gathered}
\end{equation}
where $w,x,\eta,\beta$ are functions of $t$, $\phi\ne 0$ is a constant that measures the angle between $m_1$ and $m_2$, and $w^2+x^2-\eta^2=-1$. This candidate solution
maintains a rectangular orbit since, if $\nu_{ij}={\bf q}_i\cdot{\bf q}_j$, then
$$
\begin{gathered}
\nu_{12}=\nu_{34}=1-2\eta^2=-2w^2-2x^2-1,\\
\nu_{13}= \nu_{24}=\eta^2-1-\eta^2\cosh\phi=w^2+x^2-(w^2+x^2+1)\cosh\phi,\\
\nu_{14}=\nu_{23}=-\eta^2+1-\eta^2\cosh\phi-w^2-x^2-(w^2+x^2+1)\cosh\phi.
\end{gathered}
$$
Notice also that
$$
\dot\beta=\frac{a}{4m\eta^2}=\frac{a}{4m(w^2+x^2+1)}.
$$
Criterion 4 in \cite{Diacu05} shows that the equations of motion take the form
\begin{equation}\label{rnhdiff1}
\begin{cases}
\ddot w_1=K(w,x,\dot w,\dot x)w_1\cr
\ddot w_2=K(w,x,\dot w,\dot x)w_2\cr
\ddot x_1=K(w,x,\dot w,\dot x)x_1\cr
\ddot x_2=K(w,x,\dot w,\dot x)x_2\cr
\eta\ddot\beta+2\dot\eta\dot\beta=\frac{m\eta\sinh\phi}{[(\eta^2-1-\eta^2\cos\phi)^2-1]^{3/2}}+\frac{m\eta\sinh\phi}{[(\eta^2-1+\eta^2\cos\phi)^2-1]^{3/2}}\cr
\eta\ddot\beta+2\dot\eta\dot\beta=-\frac{m\eta\sinh\phi}{[(\eta^2-1-\eta^2\cos\phi)^2-1]^{3/2}}-\frac{m\eta\sinh\phi}{[(\eta^2-1+\eta^2\cos\phi)^2-1]^{3/2}},
\end{cases}
\end{equation}
where 
$$
K(w,x,\dot w,\dot x)=m\bigg[\frac{1+\nu_{12}}{(\nu_{12}^2-1)^{3/2}}+\frac{1+\nu_{13}}{(\nu_{13}^2-1)^{3/2}}+\frac{1+\nu_{14}}{(\nu_{14}^2-1)^{3/2}}\bigg ]+
$$
$$
\frac{\dot w^2 + \dot x^2 + (w \dot x - \dot w x)^2}{w^2+x^2+1}+\frac{a^2}{16m^2(w^2+x^2+1)}.
$$
Notice that for the above expression of $\dot\beta$, the left hand sides of the last two equations in \eqref{rnhdiff1} are zero, so the right hand sides must be zero too. But that happens only for $\phi=0$, which does not correspond to a proper hyperbolic rectangle. This remark completes the proof.
\end{proof}

%%%%%%%%%
\begin{theorem}[{\bf Negative elliptic-hyperbolic rectangular rotopulsators}] In the curved $4$-body problem in $\mathbb H^3$, there are no negative elliptic-hyperbolic rectangular rotopulsators. 
\end{theorem}
\begin{proof}
Consider a solution candidate (see Definition 5 in \cite{Diacu05}) of the form
\begin{equation}\label{rnehsol}
\begin{gathered}
{\bf q}=({\bf q}_1, {\bf q}_2, {\bf q}_3, {\bf q}_4),\ {\bf q}_i=(w_i,x_i,y_i,z_i),\ i=1,2,3,4,\\
w_1=r\cos\alpha,\ x_1=r\sin\alpha,\ y_1=\eta\sinh\beta,\ z_1=\eta\cosh\beta,\\
w_2=-r\cos\alpha,\ x_2=-r\sin\alpha,\ y_2=\eta\sinh\beta,\ z_2=\eta\cosh\beta,\\
w_3=r\cos\alpha,\ x_3=r\sin\alpha,\ y_3=\eta\sinh(\beta+\phi),\ z_3=\eta\cosh(\beta+\phi),\\
w_4=-r\cos\alpha,\ x_4=-r\sin\alpha,\ y_4=\eta\sinh(\beta+\phi),\ z_4=\eta\cosh(\beta+\phi),\\
\end{gathered}
\end{equation}
with $r, \eta, \alpha, \beta$ functions of $t$, $r^2-\eta^2=-1$, $z_i \geq 1$, $i=1,2,3,4$, and $\phi\ne 0$. This candidate solution maintains a rectangular configuration since, if $\delta_{ij}={\bf q}_i\cdot{\bf q}_j$, then
$$
\begin{gathered}
\delta_{12}=\delta_{34}=-2r^2-1, \ \
\delta_{13}=\delta_{24}=r^2-\eta^2\cosh\phi,\\
\delta_{14}=\delta_{23}=-r^2-\eta^2\cosh\phi.\\
\end{gathered}
$$
Also notice that
$$
\dot\alpha=\frac{d_1}{4mr^2},\ \dot\beta=\frac{d_2}{4m\eta^2}.
$$
The equations of motion reduce to the system
\begin{equation}\label{rnehdiff4}
\begin{cases}
\dot r=\mu \cr
\dot \mu= r\Big\{(1+r^2)\big[\frac{d_1^2}{16m^2 r^4}-\frac{d_2^2}{16m^2(1 + r^2)^2}\big] +\frac{\dot r^2}{1+r^2} +
\frac{m[-2(1+r^2)]}{(\delta_{12}^2-1)^{3/2}} +\cr  {}\hspace{1cm}\frac{m[(1+r^2)(1-\cosh \phi)]}{(\delta_{13}^2-1)^{3/2}} + \frac{m[-(1+r^2)(1+\cosh \phi)]}{(\delta_{14}^2-1)^{3/2}}\Big\}\cr
r\ddot\alpha+2\dot r\dot\alpha=0\cr
\eta\ddot\beta+2\dot\eta\dot\beta=\frac{m\eta\sinh\phi}{[(r^2-\eta^2\cos\phi)^2-1]^{3/2}}+\frac{m\eta\sinh\phi}{[(r^2+\eta^2\cos\phi)^2-1]^{3/2}}.\cr
\end{cases}
\end{equation}
To prove the nonexistence of the rectangular orbits it is enough to focus on the last equation above. The expression of $\dot\beta$ shows that the left hand side of this equation is zero. But the right hand side is zero only for $\phi=0$, which does not correspond to a proper hyperbolic rectangle. This remark completes the proof.
\end{proof}

\noindent{\bf Acknowledgements.} The authors acknowledge the partial support of a Discovery Grant (Florin Diacu) and an USRA Fellowship (Brendan Thorn), both awarded by NSERC of Canada.

\end{document}